\documentclass[a4paper,11pt]{amsart}

\usepackage{latexsym, amssymb, amsmath, amsthm, mathrsfs, bbm, pstricks, ifthen}

\usepackage[all, knot]{xy}
\xyoption{arc}

\usepackage[colorlinks,linkcolor=blue,citecolor=blue]{hyperref}

\allowdisplaybreaks[4]

\newcommand{\Dchaintwo}[4]{
\rule[-3\unitlength]{0pt}{8\unitlength}
\begin{picture}(14,5)(0,3)
\put(1,2){\ifthenelse{\equal{#1}{l}}{\circle*{2}}{\circle{2}}}
\put(2,2){\line(1,0){10}}
\put(13,2){\ifthenelse{\equal{#1}{r}}{\circle*{2}}{\circle{2}}}
\put(1,5){\makebox[0pt]{\scriptsize #2}}
\put(7,4){\makebox[0pt]{\scriptsize #3}}
\put(13,5){\makebox[0pt]{\scriptsize #4}}
\end{picture}}

\def \To{\longrightarrow}
\def \dim{\operatorname{dim}}
\def \D{\operatorname{d}}
\def \gr{\operatorname{gr}}
\def \Hom{\operatorname{Hom}}

\def \Vec{\operatorname{Vec}}
\def \comod{\operatorname{comod}}
\def \mod{\operatorname{mod}}
\def \C{\mathbb{C}}
\def \N{\mathbb{N}}
\def \M{\mathrm{M}}

\def \D{\Delta}

\def \e{\varepsilon}
\def \M{\mathrm{M}}
\def \Z{\mathbb{Z}}

\def \k{\mathbbm{k}}
\def \1{\mathbf{1}}
\def \id{\operatorname{id}}

\def \GL{\operatorname{GL}}

\numberwithin{equation}{section}

\newtheorem{theorem}{Theorem}[section]
\newtheorem{lemma}[theorem]{Lemma}
\newtheorem{proposition}[theorem]{Proposition}
\newtheorem{corollary}[theorem]{Corollary}
\newtheorem{definition}[theorem]{Definition}
\newtheorem{example}[theorem]{Example}
\newtheorem{conjecture}[theorem]{Conjecture}
\newtheorem{remark}[theorem]{Remark}

\begin{document}

\title[Nondiagonal finite quasi-qantum groups]{On nondiagonal finite quasi-qantum groups \\ over finite abelian groups$^\dag$}\thanks{\tiny $^\dag$Supported by NSFC 11471186, 11571199 and 11701468.}



\author[H.-L. Huang]{Hua-Lin Huang}
\address{Huang: School of Mathematical Sciences, Huaqiao University, Quanzhou 362021, China}
\email{hualin.huang@hqu.edu.cn}

\author[Y. Yang]{Yuping Yang*}\thanks{*Corresponding author.}
\address{Yang: School of Mathematics and statistics, Southwest University, Chongqing 400715, China}
\email{yupingyang@swu.edu.cn}

\author[Y. Zhang]{Yinhuo zhang}
\address{Zhang: Department of Mathematics and statistics, University of Hasselt, Universitaire Campus, 3590 Diepenbeek, Belgium}
\email{yinhuo.zhang@uhasselt.be}

\date{}
\maketitle

\begin{abstract}
In this paper, we initiate the study of nondiagonal finite quasi-quantum groups over finite abelian groups. We mainly study the Nichols algebras in the twisted Yetter-Drinfeld module category $_{\k G}^{\k G}\mathcal{YD}^\Phi$ with $\Phi$ a nonabelian $3$-cocycle on a finite abelian group $G.$ A complete clarification is obtained for the Nichols algebra $B(V)$ in case $V$ is a simple twisted Yetter-Drinfeld module of nondiagonal type. This is also applied to provide a complete classification of finite-dimensional coradically graded pointed coquasi-Hopf algebras over abelian groups of odd order and confirm partially the generation conjecture of pointed finite tensor categories due to Etingof, Gelaki, Nikshych and Ostrik.
\end{abstract}

\section{Introduction}
This is a further contribution to the classification problem of finite quasi-quantum groups and pointed finite tensor categories over finite abelian groups beyond some previous works \cite{HLYY, HLYY2} by the first two authors jointly with Liu and Ye. Throughout, we work over an algebraically closed field $\k$ of characteristic zero. Unless stated otherwise, in this paper all spaces, maps, (co)algebras, (co)modules, and categories, etc., are over $\k.$

 In \cite{HLYY, HLYY2}, finite quasi-quantum groups of diagonal type are classified. A key observation is that the study of such algebras can be transformed to that of finite-dimensional pointed Hopf algebras over abelian groups. The latter has been successfully developed and featured with many powerful tools such as Nichols algebras, Weyl groupoids, and arithmetic root systems, see e.g. \cite{ AS, AS2, H0, H4}. In this paper, we initiate the investigation of finite quasi-quantum groups of nondiagonal type. In the following we use some concrete notations to provide more explicit explanations.

Once and for all, let $G$ be a finite abelian group and $\Phi$ be a $3$-cocycle on $G.$ A complete understanding of the Nichols algebras in the twisted Yetter-Drinfeld module category $_{\k G}^{\k G}\mathcal{Y}\mathcal{D}^{\Phi}$ is the crux for the classification of finite-dimensional pointed coquasi-Hopf algebras. A twisted Yetter-Drinfeld module $V \in {_{\k G}^{\k G}\mathcal{Y}\mathcal{D}^{\Phi}}$ is said to be of diagonal type, if it is a direct sum of $1$-dimensional twisted Yetter-Drinfeld modules. The associated Nichols algebra $B(V)$ and the coquasi-Hopf algebra $B(V)\#\k G $ are called diagonal if $V$ is so. It is shown in \cite{HLYY2} that, all $V \in {_{\k G}^{\k G}\mathcal{Y}\mathcal{D}^{\Phi}}$ are diagonal if and only if the $3$-cocycle $\Phi$ is abelian, and if and only if there exists a bigger finite abelian group $\mathbb{G}$ with canonical projection $\pi: \mathbb{G} \to G$ such that $\pi^* (\Phi)$ is a 3-coboundary on $\mathbb{G}.$ If this is the case, then the diagonal Nichols algebras in $_{\k G}^{\k G}\mathcal{Y}\mathcal{D}^{\Phi}$ can essentially be reduced to those in  $^{\k\mathbbm{G}}_{\k\mathbbm{G}} \mathcal{YD}^{\pi^* (\Phi)},$ and thus in $^{\k\mathbbm{G}}_{\k\mathbbm{G}} \mathcal{YD}$ as $\pi^* (\Phi)$ is a 3-coboundary. So to go further beyond \cite{HLYY2}, we shall consider the case with $\Phi$ nonabelian and $V \in {_{\k G}^{\k G}\mathcal{Y}\mathcal{D}^{\Phi}}$ nondiagonal.

To this end, in principle we need to develop a theory for the Nichols algebras of semisimple twisted Yetter-Drinfeld modules. The Hopf version of such a theory was developed in \cite{AHS}. However, at present it seems not easy to extend this theory to the quasi-Hopf case directly. As a trial step, firstly we study the Nichols algebras of semisimple twisted Yetter-Drinfeld modules with few summands. It turns out that if the number of summands is less than or equal to $2,$ then we are able to make a connection from this to the diagonal case. The main idea is to consider the support groups of such easy Yetter-Drinfeld modules and carry out the base group change as in our previous works \cite{HLYY, HLYY2}. More precisely, if $V \in {_{\k G}^{\k G}\mathcal{Y}\mathcal{D}^{\Phi}}$ is nondiagonal and has at most $2$ simple summands, then its support group $G_V$ is either a cyclic group or the direct product of two cyclic groups. Moreover, the Nichols algebra $B(V) \in {_{\k G}^{\k G}\mathcal{Y}\mathcal{D}^{\Phi}}$ is essentially nothing other than $B(V) \in {_{\k G_V}^{\k G_V}\mathcal{YD}^{\Phi|_{G_V}}}.$ In this situation, all $3$-cocycles on $G_V$ are abelian and then \cite{HLYY, HLYY2} can be applied.

Our first main result is a complete clarification of the Nichols algebra $B(V)$ when $V$ is a simple twisted Yetter-Drinfeld module of nondiagonal type. In particular, we provide an explicit necessary and sufficient condition on $V$ for $B(V)$ to be finite-dimensional. The same idea and process can be applied to $B(V)$ when $V$ is a direct sum of $2$ simple twisted Yetter-Drinfeld modules. As this will not provide more insights for our ultimate aim, we do not include a detailed discussion of this case. Instead, we present several simple examples to offer the reader some flavor. Surprisingly, the result on $B(V)$ with $V$ simple is already enough for us to achieve half of our final aim. Our second main result is a complete classification of finite-dimensional coradically graded pointed coquasi-Hopf algebras over abelian groups of odd order. The key observation is that $B(V) \in {_{\k G}^{\k G}\mathcal{Y}\mathcal{D}^{\Phi}}$ is infinite-dimensional for any simple nondiagonal twisted Yetter-Drinfeld module $V$ if the order of $G$ is odd. As an application, we also prove that any pointed finite tensor category over an abelian group of odd order is tensor generated by objects of length $2,$ which partially confirms the generation conjecture \cite[Conjecture 5.11.10.]{EGNO} of pointed finite tensor categories due to Etingof, Gelaki, Nikshych and Ostrik.

The paper is organized as follows. Section 2 is devoted to some preliminaries. In Section 3, we consider mainly the Nichols algebras in $ {_{\k G}^{\k G}\mathcal{Y}\mathcal{D}^{\Phi}}$ with $\Phi$ nonabelian. A full description of the Nichols algebra of a simple nondiagonal twisted Yetter-Drinfeld module is obtained. This is applied in Section 4 to the generation problem and a complete classification of finite-dimensional  pointed coquasi-Hopf algebras over finite abelian groups of odd order. Finally, in Section 5 we provide some further examples and problems of finite-dimensional  pointed coquasi-Hopf algebras over finite abelian groups of even order.

\section{Preliminaries}
In this section, we recall some necessary notions and basic facts on pointed finite tensor categories, pointed coquasi-Hopf algebras, twisted Yetter-Drinfeld modules, Nichols algebras, and arithmetic root systems. The reader is referred to \cite{EGNO, HLYY, HLYY2} for any unexplained concepts and notations.

\subsection{Pointed finite tensor categories and coquasi-Hopf algebras}
A finite tensor category is called pointed if every simple object is invertible. According to \cite{EO}, every pointed finite tensor category is tensor equivalent to the category of comodules of a finite-dimensional pointed coquasi-Hopf algebras.

Recall that a coquasi-Hopf algebra is a coalgebra $(\M,\D,\e)$ equipped with a
compatible quasi-algebra structure and a quasi-antipode. Namely,
there exist two coalgebra homomorphisms $$m: \M \otimes \M \To \M, \ a
\otimes b \mapsto ab \quad \text{and} \quad \mu: \k \To \M,\ \lambda \mapsto \lambda
1_\M,$$ a convolution-invertible map $\Phi: \M^{\otimes 3} \To \k$
called associator, a coalgebra antimorphism $S: \M \To \M$ and two
functions $\alpha,\beta: \M \To \k$ such that for all $a,b,c,d \in
\M$ the following equalities hold:
\begin{eqnarray*}
&a_1(b_1c_1)\Phi(a_2,b_2,c_2)=\Phi(a_1,b_1,c_1)(a_2b_2)c_2,\\
&1_\M a=a=a1_\M, \\
&\Phi(a_1,b_1,c_1d_1)\Phi(a_2b_2,c_2,d_2) =\Phi(b_1,c_1,d_1)\Phi(a_1,b_2c_2,d_2)\Phi(a_2,b_3,c_3),\\
&\Phi(a,1_\M,b)=\e(a)\e(b). \\
&S(a_1)\alpha(a_2)a_3=\alpha(a)1_\M, \quad a_1\beta(a_2)S(a_3)=\beta(a)1_\M, \label{2.11} \\
&\Phi(a_1,S(a_3),a_5)\beta(a_2)\alpha(a_4)
=\Phi^{-1}(S(a_1),a_3, S(a_5)) \alpha(a_2)\beta(a_4)=\e(a).
 \end{eqnarray*}
The triple $(S,\alpha,\beta)$ is called a quasi-antipode. $M$ is called a {\bf pointed coquasi-Hopf algebra} if $(\M,\D,\e)$ is a pointed coalgebra, i.e., every simple comodule of $M$ is $1$-dimensional.

Let $C$ be a coalgebra, the coradical $C_0$ of $C$ is the sum of all simple subcoalgebras of $C$.
Fix a coalgebra $C$ with coradical $C_0$, define $C_n$ inductively as follows: for each $n\geq 1$, define
$$C_n=\D^{-1}(C\otimes C_{n-1}+C_0\otimes C).$$
Then we get a filtration $C_0\subset C_1\subset \cdots \C_n\subset \cdots, $ which is called the coradical filtration of $C$. A coquasi-Hopf algebra also has a coradical filtration since every coquasi-Hopf algebra is a coalgebra.

Given a coquasi-Hopf algebra $(\M,\D, \e,
m, \mu, \Phi,S,\alpha,\beta),$ let $\{\M_n\}_{n \ge 0}$ be its
coradical filtration, and let $$\gr \M = \M_0 \oplus \M_1/\M_0 \oplus \M_2/\M_1
\oplus \cdots,$$ the corresponding coradically graded coalgebra. Then naturally
$\gr \M$ inherits from $\M$ a graded coquasi-Hopf algebra structure. The
corresponding graded associator $\gr\Phi$ satisfies
$\gr\Phi(\bar{a},\bar{b},\bar{c})=0$ for all homogeneous
$\bar{a},\bar{b},\bar{c} \in \gr \M$ unless they all lie in $\M_0.$
Similar conditions hold for $\gr\alpha$ and $\gr\beta.$ A coquasi-Hopf algebra $M$ is called {\bf coradically graded} if $M\cong \gr(M)$ as coquasi-Hopf algebras.

Here is an example with some useful terms and notations for our later investigations.

\begin{example}
Let $G$ be a group. Clearly the group algebra $\k G$ is a Hopf algebra with $\D(g)=g\otimes g,\; S(g)=g^{-1}$ and
$\e(g)=1$ for any $g\in G$. Let $\omega$ be a normalized $3$-cocycle on $G$, i.e.
\begin{eqnarray}
&\omega(ef,g,h)\omega(e,f,gh)=\omega(e,f,g)\omega(e,fg,h)\omega(f,g,h),\\
&\omega(f,1,g)=1
\end{eqnarray}
for all $e,f,g,h \in G$. By linearly extending, $\omega \colon (\k G)^{\otimes 3} \to \k$ becomes a
 convolution-invertible map. Define two linear functions
 $\alpha,\beta \colon \k G \to \k$ by \[ \alpha(g):=\e(g) \quad \text{and} \quad \beta(g):=\frac{1}{\omega(g,g^{-1},g)} \]
 for any $g\in G$. Then $kG$ together with these $\omega, \ \alpha$ and $\beta$
makes a coquasi-Hopf algebra, which will be written as $(\k G,\omega)$ in the following. By definition, the Gr-category $\Vec_G^\omega$ is just the category of comodules of $(\k G,\omega)$.
\end{example}

 It is well known that a pointed fusion category over $\k$ is equivalent to a Gr-category, see \cite{EGNO} for details. The crux to determine all the pointed fusion categories is to give a complete list of the representatives of the $3$-cohomology classes in $\operatorname{H}^3(G,\k^*)$ for all groups $G.$ However, when $G$ is a finite abelian group, the problem is solved in \cite{HLYY2}, and a list of the representatives of $\operatorname{H}^3(G,\k^*)$ can be given as follows.

Let $\N$ denote the set of nonnegative integers, $\Z$ the ring of integers, and $\Z_m$ the cyclic group of order $m.$ Any finite abelian group $G$ is of the form $\mathbb{Z}_{m_{1}}\times\cdots \times\mathbb{Z}_{m_{n}}$ with $m_j\in \mathbb{N}$
for $1\leq j\leq n.$  Denote by $\mathcal{A}$ the set of all $\N$-sequences
\begin{equation}\label{cs}(c_{1},\ldots,c_{l},\ldots,c_{n},c_{12},\ldots,c_{ij},\ldots,c_{n-1,n},c_{123},
\ldots,c_{rst},\ldots,c_{n-2,n-1,n})\end{equation}
such that $ 0\leq c_{l}<m_{l}, \ 0\leq c_{ij}<(m_{i},m_{j}), \ 0\leq c_{rst}<(m_{r},m_{s},m_{t})$ for $1\leq l\leq n, \ 1\leq i<j\leq n, \ 1\leq r<s<t\leq n$, where $c_{ij}$ and $c_{rst}$ are ordered in the lexicographic order of their indices. We denote by $\underline{\mathbf{c}}$ the sequence \eqref{cs}  in the following. Let $g_i$ be a generator of $\mathbb{Z}_{m_{i}}, 1\leq i\leq n$. For any $\underline{\mathbf{c}}\in \mathcal{A}$, define
\begin{eqnarray}\label{2.28}
&& \omega_{\underline{\mathbf{c}}}:\;G\times G\times G\To \k^{\ast} \notag \\
&&[g_{1}^{i_{1}}\cdots g_{n}^{i_{n}},g_{1}^{j_{1}}\cdots g_{n}^{j_{n}},g_{1}^{k_{1}}\cdots g_{n}^{k_{n}}] \mapsto \\ && \prod_{l=1}^{n}\zeta_{m_l}^{c_{l}i_{l}[\frac{j_{l}+k_{l}}{m_{l}}]}
\prod_{1\leq s<t\leq n}\zeta_{m_{t}}^{c_{st}i_{t}[\frac{j_{s}+k_{s}}{m_{s}}]}
\prod_{1\leq r<s<t\leq n}\zeta_{(m_{r},m_{s},m_{t})}^{c_{rst}k_{r}j_{s}i_{t}}. \notag
\end{eqnarray}
Here and below $\zeta_m$ stands for an $m$-th primitive root of unity.

\begin{proposition}\cite[Proposition 3.8]{HLYY2}\label{p2.26}   $\{\omega_{\underline{\mathbf{c}}} \mid \underline{\mathbf{c}}\in \mathcal{A}\}$ forms a complete set of representatives of the normalized $3$-cocycles on $G$ up to $3$-cohomology.
\end{proposition}

\subsection{Twisted Yetter-Drinfeld module categories}
The Yetter-Drinfeld module category $^H_H\mathcal{YD}$ of a quasi-Hopf algebra $H$ may be defined as the center $Z(H$-$\mod)$ of its module category
$H$-$\mod$ and it is braided tensor equivalent to the module category of the quantum double $D(H)$ of $H,$ see \cite{majid, m3} for more details.
The Yetter-Drinfeld module category of a coquasi-Hopf algebra can be defined in a dual manner, see \cite{HLYY2,HY}.

In this paper we are mainly concerned with the Yetter-Drinfeld module category of the coquasi-Hopf algebra $(\k G,\Phi)$ of a finite abelian group $G$ and a normalized
3-cocycle $\Phi$ on $G$ for our purpose. To emphasize $\Phi,$ we denote the Yetter-Drinfeld category of $(\k G,\Phi)$ as $_{\k G}^{\k G}\mathcal{Y}\mathcal{D}^{\Phi}$, and the objects in it are called {\bf twisted Yetter-Drinfeld modules}. Define
\begin{equation}
\widetilde{\Phi}_g(x,y)=\frac{\Phi(g,x,y)\Phi(x,y,g)}{\Phi(x,g,y)}
\end{equation}
 for all $g,x,y\in G$. By direct computation one can show that $\widetilde{\Phi}_g$ is a 2-cocycle on $G$. The construction of category $_{\k G}^{\k G}\mathcal{Y}\mathcal{D}^{\Phi}$ can be summarized as follows, the detailed computation can be found in \cite{HLYY2,HY}.

\begin{proposition}\label{p2.4}
A vector space $V$ is an object in $_{\k G}^{\k G}\mathcal{Y}\mathcal{D}^{\Phi}$ if and only if $V=\oplus_{g\in G}V_g$ with each $V_g$ a projective
$G$-representation with respect to the 2-cocycle $\widetilde{\Phi}_g,$ namely
\begin{equation}\label{eq2.6}
e\triangleright(f\triangleright v)=\widetilde{\Phi}_g(e,f) (ef)\triangleright v.
\end{equation}
The tensor product $V_g\otimes V_h$ is determined by
\begin{equation}\label{n3.8}
e\triangleright (X\otimes Y)=\widetilde{\Phi}_e(g,h)e\triangleright X\otimes e\triangleright Y, \ X\in V_g, \ Y\in V_h.
\end{equation}
The associativity and the braiding constraints of $_{\k G}^{\k G}\mathcal{Y}\mathcal{D}^{\Phi}$ are given respectively by
\begin{eqnarray}
&\ a_{V_e,V_f,V_g}((X\otimes Y)\otimes Z) =\Phi(e,f,g)^{-1} X\otimes (Y\otimes Z )\\
&R(X\otimes Y)=e \triangleright Y\otimes X
\end{eqnarray}
for all $X\in V_e,\  Y\in V_f,\  Z \in V_g.$
\end{proposition}

\begin{remark}\label{rm2.4}
For a simple twisted Yetter-Drinfeld module $V$ in ${_{\k G}^{\k G}\mathcal{Y}\mathcal{D}^{\Phi}}$, there exists some $g \in G$ such that $V=V_g$ and we define $g_V:=g$ in this case. Recall that a $2$-cocycle $\varphi$ on $G$ is called symmetric if $\varphi(g,h)=\varphi(h,g)$ for all $h,g\in G$. By \eqref{eq2.6},  it is not hard to show that a simple Yetter-Drinfeld module $V$ with $g_V=g$ is $1$-dimensional if and only if $\widetilde{\Phi}_g$ is symmetric.
\end{remark}

From the representatives of $3$-cocycles on abelian groups given in Proposition \ref{p2.26}, one can verify directly that
\begin{equation}
\widetilde{\Phi}_g\widetilde{\Phi}_h=\widetilde{\Phi}_{gh},\  \forall g,h\in G.
\end{equation}

The following proposition is fundamental and the proof follows from \eqref{n3.8} and the fact that $S^2=\id$.
\begin{proposition}\label{p2.5}
Suppose $V_g$ is $(G,\widetilde{\Phi}_g)$-representation, $V_h$ is a $(G,\widetilde{\Phi}_h)$-representation, then $V_g\otimes V_h$ is a $(G,\widetilde{\Phi}_{gh})$-representation.
In particular, the dual object $V_g^*$ of $V_g$ is a $(G,\widetilde{\Phi}_{g^{-1}})$-representation and $(V_g^*)^*=V_g$.
\end{proposition}

A $3$-cocycle $\Phi$ on $G$ is called an {\bf abelian $3$-cocycle} if $^{\k G}_{\k G}\mathcal{Y}\mathcal{D}^\Phi$ is pointed, i.e. each simple object of $^{\k G}_{\k G}\mathcal{Y}\mathcal{D}^\Phi$ is $1$-dimensional. Using the representatives of normalized $3$-cocycles listed in Proposition \ref{p2.26}, we can write out the representatives of abelian $3$-cocycles of a finite abelian group.

\begin{proposition}\cite[Proposition 3.14]{HLYY2}\label{p2.6}
Suppose $G=\Z_{m_1}\times \Z_{m_2}\times \cdots\times \Z_{m_n}$, $e_i$ is a generator of $\Z_{m_i}$ for all $1\leq i\leq n$, and $\Phi$ is an abelian $3$-cocycle on $G.$ Then up to cohomology $\Phi$ must be of the form
\begin{equation}\label{3.8}
\Phi(e_1^{i_1}\cdots e_n^{i_n},e_1^{j_1}\cdots e_n^{j_n},e_1^{k_1}\cdots e_N^{k_n})
=\prod_{l=1}^n\zeta_l^{c_li_l[\frac{j_l+k_l}{m_l}]}\prod_{1\leq s<t\leq n}\zeta_{m_t}^{c_{st}i_t[\frac{j_s+k_s}{m_s}]}.
\end{equation}
\end{proposition}

An object $V$ of $^{\k G}_{\k G}\mathcal{Y}\mathcal{D}^\Phi$ is said to be {\bf of diagonal type} if $V$ is a direct sum of $1$-dimensional objects. It is not hard to verify that each object of $^{\k G}_{\k G}\mathcal{Y}\mathcal{D}^\Phi$ is of diagonal type if and only if $\Phi$ is an abelian $3$-cocycle on $G$.

\subsection{Quasi-version of bosonization}\label{ss2.3}
The study of pointed coquasi-Hopf algebras may be reduced to that of Hopf algebras in twisted Yetter-Drinfeld categories. The related notions of algebras and Hopf algebras in a braided tensor category can be found in \cite{majid}.

Assume that $$\M=\bigoplus_{i \in \N} \M_{i}$$ is a coradically graded pointed coquasi-Hopf algebra over an abelian group $G$.  So $\M_0=(\k G,\Phi)$ for a $3$-cocycle $\Phi$ on $G$.
 Let $\pi:\; \M\to \M_{0}$ be
the canonical projection. Then $\M$ is a $kG$-bicomodule naturally via
$$\delta_{L}:=(\pi\otimes \id)\D,\;\;\;\;\delta_{R}:=(\id\otimes \pi)\D.$$
Thus there is a $G$-bigrading on $\M$, that is,
$$\M=\bigoplus_{g,h\in G}\;^{g}\M^{h}$$
where $^{g}\M^{h}=\{m\in \M \mid \delta_{L}(m)=g\otimes m,\;\delta_{R}(m)=m\otimes h\}$. Define the coinvariant subalgebra of $\M$ by
$$R:=\{m\in \M \mid (\id\otimes \pi)\D(m)=m\otimes 1\}.$$
Then $R=\oplus_{i\geq 0} R_i$ is a coradically graded Hopf algebra in $^{\k G}_{\k G}\mathcal{Y}\mathcal{D}^\Phi$ such that $R_0=\k$.

Conversely, let $H=\oplus_{i\geq 0} H_i$ be a coradically graded Hopf algebra in $_{G}^{G}\mathcal{YD}^{\Phi}$ such that $H_0=\k$. If $X\in H_{n}$, then we say that $X$ has length $n$. Since $H$ is a left $G$-comodule, there is a $G$-grading on $H$:
 $$H=\bigoplus_{x\in G} {^{x}H}$$
 where $^{x}H=\{X\in H|\delta_{L}(X)=x\otimes X\}.$   Hence $$H=\bigoplus_{g\in G} {^gH} =\bigoplus_{g\in G,n\in \mathbb{N}} {^{g}H_{n}}.$$
 As a convention, homogeneous elements in $H$ are denoted by capital letters, say $X, Y, Z, \dots,$ and the associated degrees are denoted by their lower cases, say $x, y, z, \dots.$ For any $X\in H$, we write its comultiplication as $$\D_{H}(X)=X_{(1)}\otimes X_{(2)}.$$

\begin{lemma}\cite[Proposition 3.3]{HY}
Keep the notations as above. We define a coquasi-Hopf algebra on $H\otimes \k G$ as follows. The product is given by
\begin{equation}
(X\otimes g)(Y\otimes h)=\frac{\Phi(xg,y,h)\Phi(x,y,g)}{\Phi(x,g,y)\Phi(xy,g,h)}X(g\triangleright Y)\otimes gh,
\end{equation}
and the coproduct is determined by
\begin{equation}
\D(X\otimes g)=\Phi(x_{(1)},x_{(2)},g)^{-1}(X_{(1)}\otimes x_{(2)}g)\otimes (X_{(2)}\otimes g).
\end{equation}
The quasi-antipode $(S,\alpha,\beta)$ is given by
\begin{eqnarray}
&S(X\otimes g)=\frac{\Phi(g^{-1},g,g^{-1})}{\Phi(x^{-1}g^{-1},xg,g^{-1})\Phi(x,g,g^{-1})}(1\otimes x^{-1}g^{-1})(S_H(X)\otimes 1), \\
&\alpha(1\otimes g)=1,\ \ \ \alpha(X\otimes g)=0,  \\
&\beta(1\otimes g)=\Phi(g,g^{-1},g)^{-1},\ \ \beta(X\otimes g)=0,
\end{eqnarray}
here $g,h\in G$ and $X,Y$ are homogeneous elements of length  $\geq 1.$
\end{lemma}
In the following, by $H\# \k G$ we denote the resulting coquasi-Hopf algebra defined on $H\otimes \k G.$
\begin{lemma}\cite[Proposition 3.4]{HY}\label{l2.7}
Let $\M$ be a coradically graded pointed coquasi-Hopf algebra over abelian group $G$ and $R$ the coinvariant subalgebra of $\M$. Then we have $R\#\k G\cong \M$ as coquasi-Hopf algebras.
\end{lemma}

\subsection{Nichols algebras and arithmetic root systems}
 Nichols algebras are the analogue of the usual symmetric algebras in more general braided tensor categories. Here for our purpose we only give the definition of a Nichols algebra in twisted Yetter-Drinfeld module categories $^{\k G}_{\k G}\mathcal{Y}\mathcal{D}^\Phi$.

Let $V$ be a nonzero object in $^{\k G}_{\k G}\mathcal{Y}\mathcal{D}^\Phi.$ By $T_{\Phi}(V)$ we denote the tensor algebra in $^{\k G}_{ \k G}\mathcal{Y}\mathcal{D}^\Phi$ generated freely by $V.$ It is clear that $T_{\Phi}(V)$ is isomorphic to $\bigoplus_{n \geq 0}V^{\otimes \overrightarrow{n}}$ as a linear space, where $V^{\otimes \overrightarrow{n}}$ means
$\underbrace{(\cdots((}_{n-1}V\otimes V)\otimes V)\cdots \otimes V).$ This induces a natural $\mathbb{N}$-graded structure on $T_{\Phi}(V).$ Define a comultiplication on $T_{\Phi}(V)$ by $\Delta(X)=X\otimes 1+1\otimes X, \ \forall X \in V,$ a counit by $\varepsilon(X)=0,$ and an antipode by $S(X)=-X.$ These provide a graded Hopf algebra structure on $T_{\Phi}(V)$ in the braided tensor category $^{\k G}_{\k G}\mathcal{Y}\mathcal{D}^\Phi.$

\begin{definition}\label{d2.9}
The Nichols algebra $B(V)$ of $V$ is defined to be the quotient Hopf algebra $T_{\Phi}(V)/I$ in $^{\k G}_{\k G}\mathcal{Y}\mathcal{D}^\Phi,$ where $I$ is the unique maximal graded Hopf ideal generated by homogeneous elements of degree greater than or equal to 2.
\end{definition}

A Nichols algebra $B(V)$ is called {\bf of diagonal type} if $V$ is a twisted Yetter-Drinfeld module of diagonal type. When $\Phi$ is an abelian $3$-cocycle on $G$, any Nichols algebra in $^{\k G}_{\k G}\mathcal{Y}\mathcal{D}^\Phi$ is of diagonal type.
In the classification of finite-dimensional pointed Hopf algebras, one is mainly concerned with Nichols algebras in $_{\k G}^{\k G}\mathcal{YD}^\Phi$ with $\Phi$ trivial. Such a Yetter-Drinfeld module category is often written in the form $_{\k G}^{\k G}\mathcal{YD}$. The Nichols algebras in $_{\k G}^{\k G}\mathcal{YD}$ are called usual Nichols algebras in order to distinguish from those in $_{\k G}^{\k G}\mathcal{YD}^\Phi$ with $\Phi$ nontrivial.

Arithmetic root systems are invariants of usual Nichols algebras of diagonal type with certain finiteness property. A complete classification of arithmetic root systems was given in \cite{H4} by Heckenberger. In \cite{HLYY,HLYY2} arithmetic root systems are applied to classify finite-dimensional pointed coquasi-Hopf algebras of diagonal type.

Suppose $B(V)$ is a usual Nichols algebra of diagonal type in $^{\k G}_{\k G}\mathcal{Y}\mathcal{D}.$ Then there is a basis $\{X_i|1\leq i\leq n\}$ of $V$ called canonical basis such that $\k X_i$ is a simple Yetter-Drinfeld module for each $1\leq i\leq n$. Suppose $\delta_L(X_i)=h_i\otimes X_i, 1\leq i\leq n$. The structure constants of $B(V)$ are $\{q_{ij}|1\leq i,j\leq n\}$ such that $h_i\triangleright X_j=q_{ij}X_j.$
Let $E=\{e_i|1\leq i\leq d\}$ be a canonical basis of $\Z^n,$ and $\chi$ be a bicharacter of $\Z^n$ determined by $\chi(e_i,e_j)=q_{ij}.$ As defined in \cite[Sec.3]{H0}, $\Delta^{+}(B(V))$ is the set of degrees of the (restricted) Poincare-Birkhoff-Witt generators counted with multiplicities and $\Delta(B(V)):= \Delta^{+}(B(V))\bigcup -\Delta^{+}(B(V))$, which is called the root system of $B(V)$. Moreover, the triple $(\Delta=\Delta(B(V)),\chi, E)$ is called an arithmetic root system of $B(V)$ if the corresponding Weyl groupoid $W_{\chi, E}$ is full and finite (see \cite[Sec.2,3]{H4}). In this case, we denote this arithmetic root system by $\Delta(B(V))_{\chi,E}$ for brevity.
If there is another arithmetic root system $\Delta_{\chi',E'},$ and an isomorphism $\tau:\Z^n\to \Z^n$ such that
$$\tau(E)=E',\ \ \ \ \chi'(\tau(e),\tau(e))=\chi(e,e),$$
$$\chi'(\tau(e_1),\tau(e_2))\chi'(\tau(e_2),\tau(e_1))=\chi(e_1,e_2)\chi'(e_2,e_1),$$ then we say that $\Delta_{\chi,E}$ and $\Delta_{\chi',E'}$ are twist equivalent.

A generalized Dynkin diagram is an invariant of arithmetic root systems, and it can determine arithmetic root systems up to twist equivalence.

\begin{definition}
The generalized Dynkin diagram of an arithmetic root system $\Delta_{\chi,E}$ is a nondirected graph $\mathcal{D}_{\chi,E}$ with the following properties:
\begin{itemize}
\item[1)]There is a bijective map $\phi$ from $I=\{ 1, 2, \dots, d \}$ to the set of vertices of $\mathcal{D}_{\chi,E}.$
\item[2)]For all $1\leq i\leq d,$ the vertex $\phi(i)$ is labelled  by $q_{ii}.$
\item[3)]For all $1\leq i,j\leq d,$ the number  $n_{ij}$ of edges between $\phi(i)$ and $\phi(j)$ is either $0$ or $1.$ If $i=j$ or $q_{ij}q_{ji}=1$ then $n_{ij}=0,$ otherwise $n_{ij}=1$ and the edge is labelled by $\widetilde{q_{ij}}=q_{ij}q_{ji}$ for all $1\leq i<j\leq n.$
\end{itemize}
\end{definition}
An arithmetic root system $\Delta_{\chi,E}$ is called \emph{connected} if and only if the corresponding generalized Dynkin diagram $\mathcal{D}_{\chi,E}$ is connected. All the generalized Dynkin diagrams of connected arithmetic root systems are listed in \cite{H4}.

\section{Nichols algebras in twisted Yetter-Drinfeld categories}
In this section, we focus on the Nichols algebras in twisted Yetter-Drinfeld module categories $ {_{\k G}^{\k G}\mathcal{Y}\mathcal{D}^{\Phi}}$ with $\Phi$ nonabelian. With a help of \cite{HLYY2}, an explicit description of the finite-dimensional Nichols algebra of a simple nondiagonal twisted Yetter-Drinfeld module is obtained.

\subsection{Some basic facts on Nichols algebras} It is known that there is an $\N$-graded structure on $B(V)$. We will show that there is actually a $\Z^l$-graded structure on $B(V)$ for $V=\oplus_{i=1}^lV_i\in {^{\k G}_{\k G}\mathcal{YD}^\Phi}$, where the $V_i$'s are simple. Let $\{e_i:1\leq i\leq l\}$ be a set of free generators of $\Z^l$. Then we have the following proposition, which in fact is a generalization of \cite[Proposition 4.2]{HLYY2}.

\begin{proposition}\label{p4.1}
There is a $\Z^{l}$-grading on the Nichols algebra $B(V)\in {_{\k G}^{\k G} \mathcal{YD}^\Phi}$ by setting $\deg V_i=e_i$.
\end{proposition}

\begin{proof}
Obviously, there is a $\Z^{l}$-grading on the tensor algebra $T_{\Phi}(V)\in {_{\k G}^{\k G} \mathcal{YD}^\Phi}$ by assigning $\deg V_i=e_i$, that is $\deg(X)=e_i$ for all $X\in V_i$, $1\leq i\leq n$. Let $I=\oplus_{i\geq 1} I_i$ be the maximal graded Hopf ideal generated by $\N$-homogeneous elements of degree greater than or equal to $2.$ To prove that $B(V)$ is $\Z^l$-graded, it suffices to prove that $I$ is $\Z^l$-graded. This will be done by induction on the $\N$-degree.

Since $I=\oplus_{i\geq 1} I_i$ is generated by $\N$-homogeneous elements of degree greater than or equal to $2$, it is obvious that $I_1=0$. Hence $I_1$ is  $\Z^l$-graded.

Now suppose that $I^k:=\oplus_{1\leq i\leq k} I_i$ is $\Z^l$-graded. We shall prove that $I^{k+1}=\oplus_{1\leq i\leq k+1} I_i$ is also $\Z^l$-graded. Let $X\in I_{k+1}$ and $X=X^1+X^2+\cdots +X^n,$ with each $X^i$ being $\Z^l$-homogenous and $X^i$ and $X^j$ having different $\Z^l$-degrees if $i\neq j.$ Write $\Delta(X^i)=X^i\otimes 1+1\otimes X^i+(X^i)_1\otimes (X^i)_2.$ Since $\Delta(X)=X\otimes 1+1\otimes X+(X)_1\otimes (X)_2,$ where $(X)_1\otimes (X)_2\in T_{\Phi}(V)\otimes I^k+I^k\otimes T_{\Phi}(V),$ i.e., $\sum (X^i)_1\otimes (X^i)_2\in T_{\Phi}(V)\otimes I^k+I^k\otimes T_{\Phi}(V).$ According to the inductive assumption, $T_{\Phi}(V)\otimes I^k+I^k\otimes T_{\Phi}(V)$ is a $\Z^l$-graded space. So each $(X^i)_1\otimes (X^i)_2\in T_{\Phi}(V)\otimes I^k+I^k\otimes T_{\Phi}(V)$ as $\Delta$ preserves $\Z^l$-degrees. If there was an $X^i\notin I_{k+1},$ then $I+\langle X^i\rangle$ is a Hopf ideal properly containing $I,$ which contradicts to the maximality of $I.$ It follows that $X^i\in I_{k+1}$ for all $1\leq i\leq n$ and hence $I^{k+1}$ is also $\Z^l$-graded by the assumption on $X.$ This completes the proof of the proposition.
\end{proof}

If $B(V)$ is a Nichols algebra in $_{\k G}^{\k G}\mathcal{YD}^\Phi$, we say that $G$ is the {\bf base group} of $B(V)$.
From Definition \ref{d2.9} we know that a Nichols algebra depends on both the base group $G$ and the $3$-cocycle $\Phi$. But sometimes we are only concerned about the braided Hopf algebra structure of a Nichols algebra. Hence we need to omit or change the base group of the Nichols algebra in the sense of the following definition.

\begin{definition}
Let $B(V)$ and $B(U)$ be Nichols algebras in $_{\k G}^{\k G}\mathcal{YD}^\Phi$  and $_{\k H}^{\k H}\mathcal{YD}^\Psi$ respectively with $\dim V=\dim U=l.$ We say that $B(V)$ is isomorphic to $B(U)$ if there is a $\Z^l$-graded linear isomorphism $\mathcal{F}:B(V)\to B(U)$ which preserves the multiplication and comultiplication.
\end{definition}

\begin{definition}
Let $V=\oplus_{i=1}^n V_i\in {^{\k G}_{\k G}\mathcal{YD}^\Phi}$ be a Yetter-Drinfeld module, where $V_i \ (1\leq i\leq n)$ are simple Yetter-Drinfeld modules. Let $g_i$ be the corresponding degree of $V_i$. Then we call the subgroup $G'=\langle g_1,\cdots,g_n\rangle,$ generated by $g_1,\ldots, g_n,$ the support group of $V$, which is denoted by $G_{V}$.
\end{definition}

\begin{lemma}\cite[Lemma 4.4]{HLYY2}\label{l3.4}
Suppose $V\in {_{\k G}^{\k G}\mathcal{YD}^\Phi}$ and $U\in { _{\k H}^{\k H}\mathcal{YD}^\Psi}$, where $H$ is a finite abelian group.  Let $G|_V$ and $H|_U$ be the support groups of $V$ and $U$ respectively. If there is a linear isomorphism $F:V\to U$ and a group epimorphism $f:G|_V\to H|_U$ such that:
\begin{eqnarray}
&\delta\circ F=(f\otimes F)\circ \delta,\label{eq4.1} \\
& F(g\triangleright v)=f(g)\triangleright F(v),\label{eq4.2}\\
& \Phi|_{G|_V}=f^*\Psi|_{H|_U}
\end{eqnarray}
for any $g\in G|_V,$ $v\in V.$ Then $B(V)$ is isomorphic to $B(U).$
\end{lemma}

\begin{corollary}\label{c3.5}
Let $B(V)$ be a Nichols algebra in $_{\k G}^{\k G}\mathcal{YD}^\Phi$, $H=G_V$ and $\Psi=\Phi|_{G_V}.$ Then there is a Yetter-Drinfeld module $U$ in $_{\k H}^{\k H}\mathcal{YD}^\Psi$ such that $B(V)\cong B(U)$.
\end{corollary}
\begin{proof}
Let $U=V$ as linear space with module and comodule structures inherited from those of $V$. Then $U$ is a Yetter-Drinfeld module in $_{\k H}^{\k H}\mathcal{YD}^\Psi$, and $B(V)$ is isomorphic to $B(U)$ by Lemma \ref{l3.4}.
\end{proof}

 Next we will introduce the twist of a Nichols algebra. Let $(V,\triangleright, \delta_{L})\in {^{\k G}_{\k G}\mathcal{YD}^\Phi}$,  and let $J$ be a $2$-cochain of $G$. Then we can define a new action $\triangleright_J$ of $G$ on $V$ by
\begin{equation}\label{3.4}
g\triangleright_J X=\frac{J(g,x)}{J(x,g)}g\triangleright X
\end{equation}
 for $X\in V$ and $g\in G.$  We denote $(V, \triangleright_{J},\delta_L)$ by $V^{J}$, and by definition we have $V^J\in {^{\k G}_{\k G}\mathcal{Y}\mathcal{D}^{\Phi\ast \partial(J)}}.$ Moreover there is a tensor equivalence $(F_J,\varphi_0,\varphi_2):\ ^{\k G}_{\k G}\mathcal{Y}\mathcal{D}^\Phi \to {^{\k G}_{\k G}\mathcal{Y}\mathcal{D}^{\Phi\ast\partial (J)}}$ which takes $V$ to $V^J$ and $$\varphi_2(U,V):(U\otimes V)^J\to U^J\otimes V^J,\ \ Y\otimes Z\mapsto J(y,z)^{-1}Y\otimes Z$$ for $Y\in U,\  Z\in V.$

 Let $B(V)$ be a usual Nichols algebra in $^{\k G}_{\k G}\mathcal{Y}\mathcal{D}.$ Then it is clear that $B(V)^J$ is a Hopf algebra in $^{\k G}_{\k G}\mathcal{Y}\mathcal{D}^{\partial J}$ with multiplication $\circ$ determined by
\begin{equation}
X\circ Y=J(x,y)XY
\end{equation} for all homogenous elements $X,Y\in B(V),$ here $x=\deg X, \ y=\deg Y$ are the associated $G$-degrees as defined in Subsection \ref{ss2.3}. Using the same terminology as for coquasi-Hopf algebras , we call $B(V)$ and $B(V)^{J}$ twist equivalent. The following fact is obvious.
\begin{lemma}\cite[Lemma 2.12]{HLYY2}\label{l3.6}
The twisting $B(V)^J$ of $B(V)$ is a Nichols algebra in $^{\k G}_{\k G}\mathcal{Y}\mathcal{D}^{\partial J}$ and $B(V)^J\cong B(V^J)$.
\end{lemma}

\subsection{Nichols algebras of diagonal type} \label{sub3.2}
In general, there are both Nichols algebras of diagonal type and of nondiagonal type in  $_{\k G}^{\k G}\mathcal{YD}^\Phi$ if $\Phi$ is nonabelian.
Recall that for a simple twisted Yetter-Drinfeld module $V$ in ${_{\k G}^{\k G}\mathcal{Y}\mathcal{D}^{\Phi}}$, there exists a $g\in G$ such that $\delta_L(X)=g \otimes X$ for all $X \in V$ and in this case we write $g_V:=g.$

\begin{example}
Let $G=\langle g_1\rangle \times \langle g_2\rangle \times\langle g_3\rangle \times\langle g_4\rangle=\Z_{m_1}\times \Z_{m_2}\times \Z_{m_3}\times \Z_{m_4}$ such that $m_i|m_j$ if $1\leq i<j\leq 4$, $\Phi$ a $3$-cocycle on $G$ given by
\begin{equation}
\Phi(g_1^{i_1}\cdots g_4^{i_4},g_1^{j_1}\cdots g_4^{j_4},g_1^{k_1}\cdots g_4^{k_4})
=\zeta_{m_1}^{k_1j_2i_3}.
\end{equation}
Let $U$ and $V$ be two simple twisted Yetter-Drinfeld modules in $_{\k G}^{\k G}\mathcal{YD}^\Phi$ such that $g_U=g_1, \ g_V=g_4$. Then $\widetilde{\Phi}_{g_1}$ is not symmetric since $\widetilde{\Phi}_{g_1}(g_2,g_3)\neq \widetilde{\Phi}_{g_1}(g_3,g_2)$, and $\widetilde{\Phi}_{g_4}$ is symmetric. Hence by Remark \ref{rm2.4}, $U$ is of nondiagonal type, while $V$ is of diagonal type.
\end{example}

The following lemma says that the study of Nichols algebras of diagonal type can always be reduced to those in a suitable twisted Yetter-Drinfeld module category $_{\k H}^{\k H}\mathcal{YD}^\Psi$ such that $\Psi$ is an abelian $3$-cocycle on $H$.

\begin{lemma}\cite[Lemma 4.1]{HLYY2}\label{l3.8}
Let $B(V)$ be a Nichols algebras of diagonal type in $_{\k G}^{\k G}\mathcal{YD}^\Phi$. Then $\Phi|_{G_V}$ is an abelian $3$-cocycle on $G_V$.
\end{lemma}

Suppose $\mathbb{G}=\mathbb{Z}_{\mathbbm{m}_1}\times\cdots\times \mathbb{Z}_{\mathbbm{m}_n} = \langle \mathbbm{g}_1 \rangle \times \cdots \langle \mathbbm{g}_n \rangle $ and $G=\mathbb{Z}_{m_1}\times\cdots\times \mathbb{Z}_{m_n} = \langle g_1 \rangle \times \cdots \langle g_n\rangle $ where $\mathbbm{m}_i=m^2_i$ for $1\leq i\leq n.$ Let
\begin{equation}\label{e4.4}\pi:\;\k\mathbbm{G}\to \k G,\;\;\;\;\mathbbm{g}_{i}\mapsto g_{i},\;\;\;\;\;1\le i\le n\end{equation}
be the canonical epimorphism. Then we have

\begin{proposition}\cite[Proposition 3.15]{HLYY2}\label{p3.9}
Suppose that $\Phi$ is an abelian $3$-cocycle on $G.$ Then $\pi^*(\Phi)$ is a $3$-coboundary on $\mathbbm{G}$.
\end{proposition}

Let $\delta_L$ and $\triangleright$ be the comodule and the module structure maps of $V\in {_{\k G}^{\k G}\mathcal{YD}^{\Phi}}$. Define
\begin{eqnarray*}
&&\rho_{L}:\;V \to \k \mathbbm{G}\otimes V,\;\;\;\;\rho_{L}=(\iota\otimes \id)\delta_{L}\\
&&\blacktriangleright:\;\k \mathbbm{G}\otimes V\to V,\;\;\;\;\mathbbm{g}\blacktriangleright Z=\pi(\mathbbm{g})\triangleright Z
\end{eqnarray*}
for all $\mathbbm{g}\in \mathbbm{G}$ and $Z\in V$. Then the following observation is immediate.
 \begin{lemma}\label{l4.9} Defined in this way, $(V, \rho_{L}, \blacktriangleright),$ denoted simply by $\widetilde{V}$ in the following, is an object in $_{\k \mathbbm{G}}^{\k \mathbbm{G}}\mathcal{YD}^{\pi^*(\Phi)}$.
 \end{lemma}

\begin{proposition}\label{p4.11}
For any Nichols algebra $B(V)\in{ _{\k G}^{\k G} \mathcal{YD}^\Phi},$ the Nichols algebra $B(\widetilde{V})\in {_{\k \mathbbm{G}}^{\k \mathbbm{G}} \mathcal{YD}^{\pi^*(\Phi)}}$ is isomorphic to $B(V).$ Moreover, if $\Phi$ is an abelian $3$-cocycle on $G$, then $B(\widetilde{V})$ is twist equivalent to a usual Nichols algebra in $ _{\k \mathbbm{G}}^{\k \mathbbm{G}} \mathcal{YD}$.
\end{proposition}
\begin{proof}
The first statement is a direct consequence of Lemma \ref{l3.4}. For the second, just note that $\pi^*(\Phi)$ is a 3-coboundary on $G$ by Proposition \ref{p3.9}. So there is a $2$-cochain $J$ on $\mathbbm{G}$ such that $\partial J=\pi^*(\Phi)$. Therefore,  $B(\widetilde{V})^{J^{-1}}\cong B(\widetilde{V}^{J^{-1}})$ is a Nichols algebra in $ _{\k \mathbbm{G}}^{\k \mathbbm{G}} \mathcal{YD}$ according to Lemma \ref{l3.6}.
\end{proof}

Thanks to the preceding proposition, each Nichols algebra of diagonal type $B(V)$ in $_{\k G}^{\k G} \mathcal{YD}^\Phi$ is twist equivalent to an ordinary Nichols algebra of diagonal type, thus has a PBW-type basis as well. So we can define root systems for Nichols algebras of diagonal type in $_{\k G}^{\k G} \mathcal{YD}^\Phi$.

\begin{definition}\index{root system}
Suppose $B(V)$ is a rank $n$ Nichols algebra of diagonal type in $_{\k G}^{\k G} \mathcal{YD}^\Phi$. Let $\Delta^{+}(B(V))$ be the set of $\Z^n$-degrees of the Poincare-Birkhoff-Witt generators counted with multiplicities and let $\Delta(B(V)):= \Delta^{+}(B(V))\bigcup -\Delta^{+}(B(V))$, which is called the root system of $B(V)$.
\end{definition}
It is obvious that if $B(V)$ is twist equivalent to an ordinary Nichols algebra $B(V')$, then $\Delta(B(V))=\Delta(B(V'))$ since the twisting does not change the $\Z^n$-degrees of PBW generators.

\begin{proposition}\label{p3.13}
Let $B(V)$ be a Nichols algebra of diagonal type in $_{\k G}^{\k G} \mathcal{YD}^\Phi$,
$\{X_i|1\leq i\leq n\}$ a canonical basis of $V$, $\delta_L(X_i)=g_i \otimes X_i, 1\leq i\leq n$. Let $(q_{ij})_{n \times n}$ be the structure constants of $V$, i.e. $
g_i\triangleright X_j=q_{ij}X_j, \ 1\leq i,j\leq n$. Let $E=\{e_1,\cdots,e_n\}$ be a basis of $\Z^n$, $\chi$ a bicharacter on $\Z^n$ such that
  $$\chi(e_i,e_j)=q_{ij},\quad 1\leq i,j\leq n.$$
  Then $\Delta(B(V))$ is finite if and only if $\Delta(B(V))_{\chi,E}$ is an arithmetic root system.
\end{proposition}
\begin{proof}

If $\Delta(B(V))_{\chi,E}$ is an arithmetic root system, then clearly $\Delta(B(V))$ is finite.

Now suppose that $\Delta(B(V))$ is finite. Since $B(V)$ is of diagonal type, it is harmless to assume that $\Phi$ is an abelian $3$-cocycle on $G$ by Lemma \ref{l3.8}. According to Proposition \ref{p4.11}, there is a $2$-cochain
$J$ on $\mathbbm{G}$ such that $B(\widetilde{V})^J=B(\widetilde{V}^J)$ is an ordinary Nichols algebra. Let $V'=\widetilde{V}^J$, $\{q'_{ij}\}$ the structure constants of $V'$, $\chi'$ a bicharacter of $\Z^n$ such that $$\chi'(e_i,e_j)=q'_{ij},\ 1\leq i,j\leq n.$$
Hence $\Delta(B(V'))_{\chi',E}$ is an arithmetic root system since $B(V')$ is an ordinary finite-dimensional Nichols algebra of diagonal type. On the other hand, we have $q'_{ii}=q_{ii},\ q'_{ij}q'_{ji}=q_{ij}q_{ji}, \ 1\leq i,j\leq n$ by \eqref{3.4}. So $\Delta(B(V'))_{\chi,E}=\Delta(B(V))_{\chi,E}$ is an arithmetic root system twist equivalent to $\Delta(B(V'))_{\chi',E}$.
This implies that $\Delta(B(V))_{\chi,E}$ is an arithmetic root system.
\end{proof}

Let $R=\oplus_{i \geq 0}R[i]$ be a coradically graded Hopf algebra in $_{\k G}^{\k G} \mathcal{YD}^\Phi$. We call $R$ connected if $R[0]=\k 1$. The following proposition is very important for our further investigation.

\begin{proposition}\label{p3.14}
Suppose that $R=\oplus_{i\geq 0} R[i]$ is a finite-dimensional connected coradically graded Hopf algebra in $_{\k G}^{\k G}\mathcal{YD}^\Phi$ such that $\Phi|_{G_{R[1]}}$ is an abelian $3$-cocycle on $G_{R[1]}$. Then $R=B(R[1])$ is a Nichols algebra.
\end{proposition}

\begin{proof}
Since $R$ is coradically graded, we have $G_{R[1]}=G_R$. Let $H=G_{R}$ and $\Psi=\Phi|_H$. It is obvious that $R$ is also a connected coradically graded Hopf algebra in $_{\k H}^{\k H}\mathcal{YD}^\Psi$. Since $\Psi$ is an abelian $3$-cocycle on $H$, we have $R=B(R[1])$ by \cite[Proposition 5.1]{HLYY2}.
\end{proof}

\subsection{The Nichols algebras of simple twisted Yetter-Drinfeld modules}
In this subsection we focus on the Nichols algebras of nondiagonal Yetter-Drinfeld modules. Note that if $\Phi$ is an abelian $3$-cocycle on $G$, then each object of $_{\k G}^{\k G}\mathcal{YD}^\Phi$ is of diagonal type. So nondiagonal Yetter-Drinfeld modules appear in $_{\k G}^{\k G}\mathcal{YD}^\Phi$ only if $\Phi$ is nonabelian. The following proposition is an immediate consequence of Propositions \ref{p2.26} and \ref{p2.6}.
\begin{proposition}\label{p4.18}
Suppose that $G$ is a cyclic group $\Z_m$ or a direct product of two cyclic groups, say $\Z_{m_1}\times \Z_{m_2}$, then all the $3$-cocycles on $G$ are abelian.
\end{proposition}

\begin{proposition}\label{p4.19}
Let $G$ be a finite abelian group, $\Phi$ a $3$-cocycle on $G$. Suppose that $B(V)$ is a Nichols algebra in $_{\k G}^{\k G}\mathcal{YD}^\Phi$, where $V$ is a simple Yetter-Drinfeld module, or a direct sum of two simple Yetter-Drinfeld modules. Then $B(V)$ is isomorphic to a Nichols algebra of diagonal type $B(V')$ in $_{\k H}^{\k H}\mathcal{YD}^\Psi$, where $H=G_V$ and $\Psi=\Phi|_{H}$.
\end{proposition}
\begin{proof}
By Proposition \ref{p2.4}, $G_V$ is either a cyclic group, or of the form $\Z_m\times \Z_n$. Hence $\Psi$ is an abelian $3$-cycycle of $H$ by Proposition \ref{p4.18}. According to Corollary \ref{c3.5}, there is a Nichols algebra $B(V')$ in $_{\k H}^{\k H}\mathcal{YD}^\Psi$ such that $B(V)\cong B(V')$. Thus $B(V')$ is of diagonal type since $\Psi$ is an abelian $3$-cocycle on $H$.
\end{proof}

According to this proposition, we can apply the theory of Nichols algebras of diagonal type to study the Nichols algebras of simple twisted Yetter-Dinfeld modules, or of a direct sum of two simple twisted Yetter-Drinfeld modules.

\begin{definition}
Let $G$ be a finite group and $\alpha$ a $2$-cocycle on $G$. An element $g\in G$ is called an $\alpha$-element if $\alpha(g,h)=\alpha(h,g)$ for all $h\in G$.
\end{definition}

\begin{proposition}\label{p4.23}
Let $G$ be a finite abelian group, $\Phi$ a $3$-cocycle on $G$. Suppose $V$ is a simple Yetter-Drinfeld module of nondiagonal type in $_{\k G}^{\k G}\mathcal{YD}^\Phi$, $g_V=g$. Then $B(V)$ is finite-dimensional if and only if $V$ is one of the following two cases:
\begin{itemize}
\item[(C1)] $g\triangleright v=-v$ for all $v\in V$;
\item[(C2)] $\dim(V)=2$ and $g\triangleright v=\zeta_3 v$ for all $v\in V$, here $\zeta_3$ is a $3$-rd primitive root of unity.
\end{itemize}
\end{proposition}

\begin{proof}
First of all, we study the simple Yetter-Drinfeld module $V$ by considering its support group. According to Proposition \ref{p2.4}, $V$ is a simple $(G,\widetilde{\Phi}_g)$-representation. We claim that $g\triangleright v=\lambda v, \ \forall v\in V$ for some nonzero constant $\lambda.$ Assume the order of $g$ is $n.$ Then
     $$\underbrace{g\triangleright (g\triangleright(\cdots (g}_n\triangleright v)\cdots ))=\prod_{i=1}^{n-1}\widetilde{\Phi}_g(g,g^i)v, \quad \forall v\in V.$$
So the action of $g$ on $V$ is diagonal. On the other hand, for each $h\in G$, we have
$$\widetilde{\Phi}_g(g,h)=\frac{\Phi(g,g,h)\Phi(g,h,g)}{\Phi(g,g,h)}=\frac{\Phi(h,g,g)\Phi(g,h,g)}{\Phi(h,g,g)}=\widetilde{\Phi}_g(h,g).$$
Hence $g$ is a $\widetilde{\Phi}_g$-element, and we have $g\triangleright(h\triangleright v)=\widetilde{\Phi}_g(g,h)gh\triangleright v=h\triangleright(g\triangleright v)$ for any $h\in G$ and $v\in V$. So the $g$-action on $V$ is a morphism of $(G,\widetilde{\Phi}_g)$-representations, and the Schur's Lemma guarantees $g\triangleright v=\lambda v, \forall v\in V$ for some nonzero constant $\lambda$.

Let $H=\langle g\rangle$ and $\Psi=\Phi|_H.$ Then by Corollary \ref{c3.5}, $B(V)$ is isomorphic to a Nichols algebra $B(V')$ in $_{\k H}^{\k H}\mathcal{YD}^\Psi$. Thus in the following it is enough to consider $V'$ and $B(V')$ instead. Let $\dim(V)=n.$ The structure constants $(q_{ij})$ of $V'$ are given by
$$q_{ij}=\lambda, \quad 1\leq i,j\leq n,$$
since $g\triangleright v=\lambda v, \ \forall v\in V'$. Let $E=\{e_1,\cdots,e_n\}$ be a basis of $\Z^n$, $\chi$ a bicharacter on $\Z^n$ such that  $$\chi(e_i,e_j)=q_{ij}=\lambda,\ 1\leq i,j\leq n.$$

If the simple Yetter-Drinfeld module $V,$ and so $V',$ satisfies either (C1) or (C2), then clearly $B(V'),$ and so $B(V),$ is finite-dimensional, see e.g. \cite{H4}. Now we prove the converse. Suppose $B(V')$ is finite-dimensional. Then $\lambda\neq 1$ and  $\Delta(B(V'))_{\chi,E}$ is an arithmetic root system by Proposition \ref{p3.13}. If $\lambda= -1$, then we have $q_{ii}=-1$ and $q_{ij}q_{ji}=1$ for all $1\leq i\neq j\leq n$. In this case, $V'$ satisfies the condition $\mathrm{C1}$ and so does $V$. Now assume $\lambda \neq -1.$ Then the arithmetic root system associated to $B(V')$ is connected since $q_{ij}q_{ji}=\lambda^2\neq 1$ for all $1\leq i\neq j\leq n$. By a careful check up on the complete classification of arithmetic root systems in \cite{H4}, one can easily conclude that the generalized Dynkin diagram of $\Delta(B(V'))_{\chi,E}$ must be as follows:
\[ {\setlength{\unitlength}{1mm} \Dchaintwo{}{$\zeta_3$}{$\zeta^{-1}_3$}{$\zeta_3$}} \quad .\] This forces $\lambda=\zeta_3$ and $n=2.$ Thus $V',$ and so $V,$ satisfies the condition $\mathrm{C2}$.
\end{proof}

In the following, we give two examples of simple Yetter-Drinfeld modules of nondiagonal type satisfying conditions $\mathrm{C1}$ and $\mathrm{C2}$ respectively. The associated nondiagonal Nichols algebras are finite-dimensional.

Recall that, if $\varphi$ is a $2$-cocycle on an abelian group $G$, then a map $\rho \colon G\to \GL(V)$ is a $(G,\varphi)$-representation if and only if
\begin{equation}\label{4.10}
\rho(1)=\id_V,\ \rho(g)\rho(h)=\frac{\varphi(g,h)}{\varphi(h,g)}\rho(h)\rho(g),\ \forall g,h\in G.
\end{equation}

\begin{example}\label{e3.19}
Let $G=\Z_2\times \Z_2\times \Z_2=\langle e_1\rangle\times \langle e_2\rangle \times \langle e_3\rangle$, $\Phi$ a $3$-cocycle on $G$ given by
$$\Phi(e_1^{i_1}e_2^{i_2}e_3^{j_3},e_1^{j_1}e_2^{j_2}e_3^{i_3},e_1^{k_1}e_2^{k_2}e_3^{k_3})=(-1)^{i_3j_2k_1}.$$
For a $2$-dimensional $\k$-vector space $V$ with a fixed basis $\{X_1,X_2\}$,
define $\rho \colon G \to \GL(V)$ by
\begin{eqnarray*}
&&\rho(1)=\left(
            \begin{array}{cc}
              1 & 0 \\
              0 & 1 \\
            \end{array}
          \right), \quad
\rho(e_1)=\left(
            \begin{array}{cc}
              -1 & 0 \\
              0 & -1 \\
            \end{array}
          \right), \quad
\rho(e_2)=\left(
         \begin{array}{cc}
              1 & 0 \\
              0 & -1 \\
            \end{array}
          \right),\\
&&\rho(e_3)=\left(
            \begin{array}{cc}
              0 & 1 \\
              1 & 0 \\
            \end{array}
          \right), \quad
 \rho(e_1e_2)=\left(
            \begin{array}{cc}
              -1 & 0 \\
              0 & 1 \\
            \end{array}
          \right), \quad
\rho(e_1e_3)=\left(
         \begin{array}{cc}
              0 & -1 \\
              -1 & 0 \\
            \end{array}
          \right),\\
&&\rho(e_2e_3)=\left(
            \begin{array}{cc}
              0 & -1 \\
              1 & 0 \\
            \end{array}
          \right), \quad
\rho(e_1e_2e_3)=\left(
            \begin{array}{cc}
              0 & 1 \\
              -1 & 0 \\
            \end{array}
          \right).
\end{eqnarray*}
Then one can verify that $\rho$ satisfies \eqref{4.10}, hence $(V,\rho)$ is a $(G,\widetilde{\Phi}_{e_1})$-representation. According to Proposition \ref{p2.4}, $V$ is a simple Yetter-Drinfeld module in $_{\k G}^{\k G}\mathcal{YD}^{\Phi}$ such that $g_V=e_1$ and $e_1\triangleright v=-v$ for all $v\in V$.

\end{example}
\begin{example}\label{e3.20}
Let $G=\Z_6\times \Z_6\times \Z_6=\langle e_1\rangle \times \langle e_2\rangle \times \langle e_3\rangle$, $\Phi$ a $3$-cocycle on $G$ given by
$$\Phi(e_1^{i_1}e_2^{i_2}e_3^{j_3},e_1^{j_1}e_2^{j_2}e_3^{i_3},e_1^{k_1}e_2^{k_2}e_3^{k_3})=(-1)^{i_3j_2k_1}.$$
Let $V$ be a $2$-dimensional $\k$-vector space with a fixed basis. Define a map $\rho \colon G\to \GL(V)$, with respect to the fixed basis of $V$, by
\begin{equation*}
\rho(e_1^{i_1}e_2^{i_2}e_3^{i_3})=(-1)^{i_2i_3}\left(
                                                 \begin{array}{cc}
                                                   \zeta_3^{i_1+i_2} & 0 \\
                                                   0 & (-1)^{i_2}\zeta_3^{i_1+i_2} \\
                                                 \end{array}
                                               \right)
                                               \left(
                                                 \begin{array}{cc}
                                                   0 & 1 \\
                                                   1 & 0 \\
                                                 \end{array}
                                               \right)^{i_3}.
\end{equation*}
Then one can verify that $\rho$ satisfies \eqref{4.10}, hence $(V,\rho)$ is a $(G,\widetilde{\Phi}_{e_1})$-representation. Since $\widetilde{\Phi}_{e_1}(e_2,e_3)=-1\neq \widetilde{\Phi}_{e_1}(e_3,e_2)=1$,  $\widetilde{\Phi}_{e_1}$ is not  symmetric.
So all the simple $(G,\widetilde{\Phi}_{g_1})$-representations have dimension $\geq 2$ by Remark \ref{rm2.4}. This implies that $(V,\rho)$ is a simple $(G,\widetilde{\Phi}_{e_1})$-representation.
By Proposition \ref{p2.4}, $V$ is a simple Yetter-Drinfeld module in $_{\k G}^{\k G}\mathcal{YD}^{\Phi}$ such that $g_V=e_1$ and $g_1\triangleright v=\zeta_3 v$ for all $v\in V$.
\end{example}

\section{Finite quasi-quantum groups over abelian groups of odd order}

In this section we provide a complete classification of finite-dimensional coradically graded pointed coquasi-Hopf algebras over abelian groups of odd order. This is also applied to the classification theory of pointed finite tensor categories.  In particular, we give a partial answer to the following
\begin{conjecture}\cite[Conjecture 5.11.10.]{EGNO}
A pointed finite tensor category is tensor generated by objects of length $2$.
\end{conjecture}
 This conjecture is due to Etingof, Gelaki, Nikshych and Ostrik, hence will be called EGNO's conjecture in the following. It is a natural generalization of the well known Andruskiewitsch-Schneider conjecture \cite[Conjecture 1.4]{AS}. Our main classification result on finite-dimensional pointed coquasi-Hopf algebras will induce the following

\begin{theorem}\label{t4.1}
Suppose that $\mathcal{C}$ is a pointed finite tensor category with $G(\mathcal{C})$ an abelian group of odd order. Then $\mathcal{C}$ is tensor generated by objects of length $2$.
\end{theorem}

\subsection{Some preparations}
In this subsection, we will give some important properties of twisted Yetter-Drinfeld modules, and prove that each finite-dimensional Nichols algebra in $_{\k G}^{\k G}\mathcal{YD}^\Phi$ must be of diagonal type if the order $|G|$ of $G$ is odd. We need the following two propositions.
\begin{proposition}\label{p5.2}
Suppose that $V$ and $W$ are two simple objects in $_{\k G}^{\k G}\mathcal{YD}^\Phi$ such that $g_V=g_W.$ Then $\dim(V)=\dim(W)$.
\end{proposition}
\begin{proof}
By $\mathcal{C}$ we denote the tensor category $_{\k G}^{\k G}\mathcal{YD}^\Phi$.
Suppose $g_V=g_W=g.$ Then $V,W$ are $(G,\widetilde{\Phi}_g)$-representations. Let $W^*$ be the dual object of $W$, which is a $(G,\widetilde{\Phi}_{g^{-1}})$-representation. Hence $V\otimes W^*$ is an ordinary representation of $G$ by Proposition \ref{p2.5}. Note that $G$ is abelian, so simple $G$-representations are $1$-dimensional. Therefore we may take a $1$-dimensional subobject of $V\otimes W^*,$ say $K$. It follows from
\begin{equation}
0\neq \Hom_{\mathcal{C}}(K,V\otimes W^*)=\Hom_{\mathcal{C}}(K\otimes W,V)
\end{equation}
that $\dim(V)=\dim(W)$ since $V$ and $W$ are simple objects.
\end{proof}

\begin{proposition}\label{p5.3}
For each simple object $V$ in $_{\k G}^{\k G}\mathcal{YD}^\Phi$, we have $\dim(V) \mid |G|$.
\end{proposition}

\begin{proof}
According to Proposition \ref{p2.4}, $V$ is a simple Yetter-Drinfeld module in $_{\k G}^{\k G}\mathcal{YD}^\Phi$ with $g_V=g$ if and only if $V$ is a simple $(G,\widetilde{\Phi}_g)$-representation.
By $\k[G]_{\widetilde{\Phi}_g}$ we denote the twisted group algebra of $G$, i.e. the algebra with a basis $\{h|h\in G\}$ and product determined by
$$f\cdot h=\widetilde{\Phi}_g(f,h)fh,\quad \forall f,h\in G.$$
Note that the representation category of $\k[G]_{\widetilde{\Phi}_g}$ is equivalent to the category of projective representations of $G$ with respect to $\widetilde{\Phi}_g$. Let $\{V^i|1\leq i\leq m\}$ be a set of iso-classes of simple representations of $\k[G]_{\widetilde{\Phi}_g}$.
By Proposition \ref{p5.2}, all simple representations of $\k[G]_{\widetilde{\Phi}_g}$ have the same dimension, which will be denoted by $n$. Since $\k[G]_{\widetilde{\Phi}_g}$ is a semisimple algebra, we have $mn^2=\dim(\k[G]_{\widetilde{\Phi}_g})=|G|$, hence $n \mid |G|$.
\end{proof}

Now we will prove two technical lemmas, which are necessary for our next exploration.

\begin{lemma}\label{l5.4}
Suppose that $\{g_1,g_2\cdots,g_n\}\subset G$ generates $G$, and $\widetilde{\Phi}_{g_i}(g_j,g_k)=\widetilde{\Phi}_{g_i}(g_k,g_j)$ for all $1\leq i,j,k\leq n.$ Then $\Phi$ is an abelian $3$-cocycle on $G$.
\end{lemma}
\begin{proof}
Suppose $G=\langle e_1\rangle \times \cdots \times  \langle e_n\rangle$ and $m_i=|e_i|$ for $1\leq i\leq n$. By Proposition \ref{p2.26}, we can assume that $\Phi$ is of the form \eqref{2.28}. Let $\Phi=\Psi\Gamma$, where $\Psi$ and $\Gamma$ are $3$-cocycles on $G$ given by
\begin{eqnarray}
&&\Psi(e_{1}^{i_{1}}\cdots e_{n}^{i_{n}},e_{1}^{j_{1}}\cdots e_{n}^{j_{n}},e_{1}^{k_{1}}\cdots e_{n}^{k_{n}}) =\prod_{l=1}^{n}\zeta_{m_l}^{c_{l}i_{l}[\frac{j_{l}+k_{l}}{m_{l}}]}
\prod_{1\leq s<t\leq n}\zeta_{m_{t}}^{c_{st}i_{t}[\frac{j_{s}+k_{s}}{m_{s}}]},\label{5.2}\\
&&\Gamma(e_{1}^{i_{1}}\cdots e_{n}^{i_{n}},e_{1}^{j_{1}}\cdots e_{n}^{j_{n}},e_{1}^{k_{1}}\cdots e_{n}^{k_{n}})=
\prod_{1\leq r<s<t\leq n}\zeta_{(m_{r},m_{s},m_{t})}^{c_{rst}k_{r}j_{s}i_{t}}.\label{5.3}
\end{eqnarray}
Since $\Psi$ is an abelian $3$-cocycle on $G$, $\widetilde{\Psi}_{g_i}(g_j,g_k)=\widetilde{\Psi}_{g_i}(g_j,g_k)$ for all $1\leq i,j,k\leq n$.
Thus $\widetilde{\Phi}_{g_i}(g_j,g_k)=\widetilde{\Phi}_{g_i}(g_j,g_k)$ implies
\begin{equation}\label{5.4}
\widetilde{\Gamma}_{g_i}(g_j,g_k)=\widetilde{\Gamma}_{g_i}(g_j,g_k)
\end{equation} for $1\leq i,j,k\leq n$.
From \eqref{5.3}, it follows that
\begin{eqnarray*}
\widetilde{\Gamma}_{ef}(g,h)&=&\widetilde{\Gamma}_{e}(g,h)\widetilde{\Gamma}_{f}(g,h),\\
\widetilde{\Gamma}_{e}(fg,h)&=&\widetilde{\Gamma}_{e}(f,h)\widetilde{\Gamma}_{e}(g,h),\\
\widetilde{\Gamma}_{e}(f,gh)&=&\widetilde{\Gamma}_{e}(f,g)\widetilde{\Gamma}_{e}(f,h)\\
\end{eqnarray*}
for all $e,f,g,h\in G$.
So \eqref{5.4} implies
\begin{equation*}
\widetilde{\Gamma}_f(g,h)=\widetilde{\Gamma}_f(h,g), \ \forall f,g,h\in G,
\end{equation*} since $G=\langle g_1,\cdots g_n\rangle$.
Hence $c_{rst}=0 $ follows from $$\widetilde{\Gamma}_{e_r}(e_s,e_t)=\widetilde{\Gamma}_{e_r}(e_t,e_s),\ 1\leq r<s<t\leq n,$$ and $\Phi=\Psi$ is an abelian $3$-cocycle on $G$.
\end{proof}

\begin{lemma}\label{l5.5}
Let $g_1,g_2,g_3$ be elements in $G.$ Then the following three identities
\begin{eqnarray}
\widetilde{\Phi}_{g_1}(g_2,g_3)&=&\widetilde{\Phi}_{g_1}(g_3,g_2),\label{5.5}\\
\widetilde{\Phi}_{g_2}(g_1,g_3)&=&\widetilde{\Phi}_{g_2}(g_3,g_1),\label{5.6}\\
\widetilde{\Phi}_{g_3}(g_1,g_2)&=&\widetilde{\Phi}_{g_3}(g_2,g_1)\label{5.7}
\end{eqnarray}
are mutually equivalent.
\end{lemma}
\begin{proof}
Suppose $\widetilde{\Phi}_{g_1}(g_2,g_3)=\widetilde{\Phi}_{g_1}(g_3,g_2)$, that is
\begin{equation}\label{5.8}
\frac{\Phi(g_1,g_2,g_3)\Phi(g_2,g_3,g_1)}{\Phi(g_2,g_1,g_3)}= \frac{\Phi(g_1,g_3,g_2)\Phi(g_3,g_2,g_1)}{\Phi(g_3,g_1,g_2)}.
\end{equation}
Multiplying the identity \eqref{5.8} with the scalar $\frac{\Phi(g_2,g_1,g_3)\Phi(g_3,g_1,g_2)}{\Phi(g_1,g_2,g_3)\Phi(g_3,g_2,g_1)}$, we get
\begin{equation*}
\widetilde{\Phi}_{g_2}(g_1,g_3)=\frac{\Phi(g_2,g_1,g_3)\Phi(g_1,g_3,g_2)}{\Phi(g_1,g_2,g_3)}= \frac{\Phi(g_2,g_3,g_1)\Phi(g_3,g_1,g_2)}{\Phi(g_3,g_2,g_1)}=\widetilde{\Phi}_{g_2}(g_3,g_1).
\end{equation*}

The equivalence between \eqref{5.7} and  \eqref{5.5}, \eqref{5.6} can be proved similarly.
\end{proof}

\begin{proposition}\label{p5.6}
Let $V=\oplus_{i=1}^n V_i$ be an object of ${_{\k G}^{\k G}\mathcal{YD}^\Phi}$, where the $V_i$'s are simple objects. Let $H=G_V$ be the support group of $V$. If $\Phi|_H$ is not an abelian $3$-cocycle on $H$, then $n\geq 3$ and at least three summands, say $V_{i_1},V_{i_2},V_{i_3}$ of $V,$ are of nondiagonal type.
\end{proposition}
\begin{proof}
Suppose $g_{V_i}=g_i$ for $1\leq i\leq n$, then $H=\langle g_1,\cdots ,g_n\rangle$. Since $\Phi|_H$ is not an abelian $3$-cocycle on $H$, there exist $i,j,k$ such that $\widetilde{\Phi}_{g_i}(g_j,g_k)\neq \widetilde{\Phi}_{g_i}(g_k,g_j)$ according to Lemma \ref{l5.4}. By Lemma \ref{l5.5}, we also have
\begin{eqnarray*}
\widetilde{\Phi}_{g_j}(g_i,g_k)&\neq& \widetilde{\Phi}_{g_j}(g_k,g_i),\\
\widetilde{\Phi}_{g_k}(g_i,g_j)&\neq& \widetilde{\Phi}_{g_k}(g_j,g_i).
\end{eqnarray*}
Hence the $2$-cocycles $\widetilde{\Phi}_{g_i},\widetilde{\Phi}_{g_j},\widetilde{\Phi}_{g_k}$ on $G$ are not symmetric, this implies $V_i,V_j,V_k$ are simple Yetter-Drinfeld modules of nondiagonal type.
\end{proof}

Now we can prove the following proposition, which says that if the order of $G$ is odd, then each finite-dimensional Nichols algebra in $_{\k G}^{\k G} \mathcal{YD}^\Phi$ must be of diagonal type.

\begin{proposition}\label{p5.9}
Let $G$ be a finite abelian group of odd order and $\Phi$ be a $3$-cocycle on $G$. Suppose that $V\in {_{\k G}^{\k G} \mathcal{YD}^\Phi}$ is not diagonal. Then $B(V)$ is infinite-dimensional.
\end{proposition}
\begin{proof}
By assumption, there is a summand $U$ of $V$ such that $U$ is a simple Yetter-Drinfeld module of nondiagonal type. Suppose that $g=g_U.$ Then there exists some $\lambda\in \k^*$ such that
 $$g\triangleright u=\lambda u, \quad \forall u\in U.$$
 Since $G$ is odd, we have $\dim(U)\neq 2$ by Proposition \ref{p5.3}. Hence, $U$ does not satisfy the condition $\mathrm{C2}$ of Proposition \ref{p4.23}.
 It is also obvious that $\lambda\neq -1$ since the order $|g|$ of $g$ is odd. Thus $U$ does not satisfy the condition $\mathrm{C1}$ of Proposition \ref{p4.23} either.
 So $B(U)$ must be infinite-dimensional, therefore so is $B(V).$
\end{proof}

\begin{corollary}\label{c4.9}
Let $G$ be a finite abelian group of odd order and $\Phi$ be a $3$-cocycle on $G$. Suppose $V \in {_{\k G}^{\k G} \mathcal{YD}^\Phi}$ such that $\Phi |_{G_V}$ is not an abelian $3$-cocycle on $G_V.$ Then $B(V)$ is infinite-dimensional.
\end{corollary}
\begin{proof}
According to Proposition \ref{p5.6}, there exists a nondiagonal simple submodule $U$ of $V.$ By Proposition \ref{p5.9}, $B(U)$ is infinite-dimensional, and so is $B(V)$.
\end{proof}

\subsection{A proof of Theorem \ref{t4.1}}
In this subsection we will prove Theorem \ref{t4.1} in several steps.
Recall that a comodule of a finite-dimensional coquasi-Hopf algebra $M$ is called cofree if it is isomorphic to $M^{\oplus n}$ as comodules for an integer $n\geq 1$. It is well-known that any finite-dimensional module of an algebra is a quotient of a free module. Dually, any finite-dimensional comodule of a coalgebra is a subcomodule of a cofree comodule.
%Let $(C,\D,\varepsilon)$ be a finite-dimensional coalgebra, $V$ a finite-dimensional $C$-comodule with a basis $\{X_1,\cdots ,X_n\}$.
%For each $1\leq i\leq n$, define a injective linear map $f_i:C\to C^{\oplus n}$
%by $$f(X)=(\underbrace{0,\cdots,0,X}_i,\cdots, 0).$$
%Let $\{m_i^j|1\leq i,j\leq n\}$ be elements of $C$ such that
%$\delta(X_i)=\sum_{j=1}^nm_i^j\otimes X_j,\ 1\leq i\leq n.$ So for each $1\leq i\leq n$, we have
%$$(\D\otimes \id)\circ \delta(X_i)=\sum_{l=1}^n\D(m_i^j)\otimes X_j= (\id\otimes \delta)\circ \delta(X_i)=\sum_{l=1}^n\sum_{j=1}^n m_i^l\otimes m_l^j\otimes X_j.$$ This implies $\D(m_i^j)=\sum_{l=1}^n m_i^l\otimes m_l^j.$
%
%Define a linear map
%\begin{eqnarray*}
%f:V&\To& C^{\oplus n}\\
%  X_i&\to& \sum_{j=1}^nf_j(m_i^j).
%\end{eqnarray*}
%Then $f$ is an injective comodule morphism follows from the commutative diagram
%\begin{equation*}
%\xymatrix@C=1.5cm{ X_i\ar[r]^{\delta}\ar[d]_{f}&\sum_{j=1}^n m_i^j\otimes X_j\ar[d]^{\id\otimes f}\\
%\sum_{i=1}^nf_j(m_i^j)\ar[r]^{\delta=\D}& \sum_{j,l=1}^nm_i^l\otimes f_j(m_l^j).}
%\end{equation*}
%Hence $V$ is a subcomodule of $C^{\oplus n}$.

\begin{proposition}\label{p5.8}
Suppose that $M$ is a finite-dimensional pointed coquasi-Hopf algebra. Then $M$ is generated by grouplike and skew-primitive elements if and only if $\comod(M)$ is tensor generated by objects of length $2$.
\end{proposition}
\begin{proof}
Suppose that $M$ is generated by grouplike and skew-primitive elements. Let $G=G(M)$, and $\{X_i|1\leq i\leq n\}$ a maximal linear independent set of skew-primitive elements. It is obvious that $(gX)g^{-1}$ is a skew-primitive element if $X$ is. So each element in $M$ can be presented as a linear combination of elements of the form $g(\cdots (X_{i_1}X_{i_2})\cdots X_{i_m})$.

Let $\mathcal{A}=\{g(\cdots (X_{i_1}X_{i_2})\cdots X_{i_m})|g\in G, (i_1,i_2,\cdots ,i_m)\in \mathcal{I}\}$ be a minimal set that cogenerates the cofree comodule $M$.
Let $V_i=\k\{g_i,X_i\},$ where $g_i$ satisfies $\delta_L(X_i)=g_i \otimes X_i, \ 1\leq i\leq n$. Let $V(g)=\k\{1, g\}$ for $g\in G$. Then it is obvious that $V_i, V(g)$ are subcomodules of $M$ of length $2$.
Let
\begin{eqnarray*}
F:M\To \oplus_{g\in G,(i_1,i_2,\cdots,i_m)\in \mathcal{I}}V(g)\otimes(\cdots (V_{i_1}\otimes V_{i_2})\otimes \cdots \otimes V_{i_m})
\end{eqnarray*}
be the linear map determined by
\begin{equation}
g(\cdots (X_{i_1}X_{i_2})\cdots X_{i_m})\to g\otimes (\cdots (X_{i_1}\otimes X_{i_2})\otimes \cdots \otimes X_{i_m}).
\end{equation}
It is obvious that $F$ is an injective comodule map. So we have proved the cofree comodule $M$ is tensor generated by objects of length $2$. This implies that each cofree comodule is tensor generated by objects of length $2$. As any finite-dimensional comodule of $M$ is a subcomodule of a cofree comodule, so $\comod(M)$ is tensor generated by objects of length $2$.

Conversely, suppose $\comod(M)$ is tensor generated by objects of length $2$. For each $g \in G$, by $S_g$ we mean a simple object such that $\delta(v)=g\otimes v$ for all $v\in S_g$. Since $\comod(M)$ is pointed, each object of length $2$ is an extension of $S_g$ by $S_h$ for some $g,h\in G$. So an object of length $2$ must be of the form $S_g\oplus S_h$, or $V=\k\{h,X\}$ where $X$ is a $g$-$h$-primitive element. Let $V(g), V_i, g\in G, 1\leq i\leq n$ be objects in $\comod(M)$ defined as the first part of the proof. Then the cofree comodule $M$ is a subquotient of object of the form
 $$\oplus V(g)\otimes(\cdots (V_{i_1}\otimes V_{i_2})\otimes \cdots \otimes V_{i_m})$$
according to the hypothesis. This implies that $M$ is generated by grouplike and skew-primitive elements.
\end{proof}

\begin{theorem}\label{t5.10}
Suppose that $R=\oplus_{i\geq 0} R[i]$ is a finite-dimensional connected coradically graded braided Hopf algebra in $_{\k G}^{\k G} \mathcal{YD}^\Phi$,  where $G$ is an abelian group of odd order and $\Phi$ is a $3$-cocycle on $G.$ Then $R=B(R[1])$.
\end{theorem}
\begin{proof}
Since $R=\oplus_{i\geq 0} R[i]$ is a connected coradically graded braided Hopf algebra, $R[0]=\k 1$ and $R[1]$ is the set of primitive elements of $R$. So there exists a canonical injective linear map $\iota \colon B(R[1])\to R$. Since $R$ is finite-dimensional, $B(R[1])$ is finite-dimensional. By Corollary \ref{c4.9}, we have that $\Phi|_{G_{R[1]}}$ is an abelian $3$-cocycle on $G_{R[1]}$. Hence $R=B(R[1])$ according to Theorem \ref{p3.14}.
\end{proof}

\noindent \textbf{Proof of Theorem \ref{t4.1}.}
Suppose that $\mathcal{C}$ is a pointed finite tensor category with $G(\mathcal{C})$ an abelian group of odd order. Then there exists a finite-dimensional pointed coquasi-Hopf algebra $M$ such that $\mathcal{C}$ is tensor equivalent to $\comod(M)$. So we only need to prove that $\comod(M)$ is tensor generated by objects of length $2$. By Proposition \ref{p5.8}, this amounts to proving that $M$ is generated by grouplike and skew-primitive elements. It is obvious that $M$ is generated by grouplike and skew-primitive elements if and only if $\gr(M)$ is so. Let $G=G(M)=G(\mathcal{C})$, $R$ the coinvariant subalgebra of $\gr(M)$. So $R$ is a finite-dimensional coradically graded braided Hopf algebra in
$_{\k G}^{\k G}\mathcal{YD}^\Phi$, where $\Phi$ is the associator of $\gr(M)$, which is actually a $3$-cocycle on $G$.
By Theorem \ref{t5.10}, $R=B(R[1])$ is a Nichols algebra. Hence $\gr(M)$ is generated by grouplike and skew-primitive elements since $\gr(M)=R\# \k G(M)$
by Lemma \ref{l2.7}. \hfill $\Box$

\subsection{The classification results}
With the help of Theorem \ref{t5.10}, we achieve the classification of coradically graded pointed coquasi-Hopf algebras and that of pointed finite tensor categories over abelian groups of odd order.

We need some new notions to present the main result. Let $\Delta_{\chi,E}$ be an arithmetic root system. For each positive root $\alpha\in \Delta$, define
$q_\alpha=\chi(\alpha,\alpha)$.
Then the height of $\alpha$ is defined by
\begin{equation}
\operatorname{ht}(\alpha)=\left\{
              \begin{array}{ll}
                |q_\alpha|, & \hbox{if $q_\alpha\neq 1$ is a root of unity;} \\
                \infty, & \hbox{otherwise.}
              \end{array}
            \right.
\end{equation}
A function $\chi:G\To \k^*$ is called a {\bf quasi-character} associated to a 2-cocycle $\omega$ on $G$ if for all $f,g\in G,$
\begin{equation}\label{3.27}
\chi(f)\chi(g)=\omega(f,g)\chi(fg),\ \ \ \chi(1)=1.
\end{equation}

It is clear that there is a quasi-character associated to $\omega$ if and only if $\omega$ is symmetric. Recall that for a fixed 3-cocycle $\Phi$ on $G$,
$\{\widetilde{\Phi}_g|g\in G\}$ gives $2$-cocycles on $G$.

\begin{definition}
Let $\chi_1,\cdots,\chi_n$ be quasi-characters of $G$ associated to $\widetilde{\Phi}_{g_1},\cdots ,\widetilde{\Phi}_{g_n}$ respectively.
We say the series $(\chi_1,\cdots,\chi_n)$ is of finite type if there is an arithmetic root system $\Delta_{\chi,E}$ of rank $n$ such that:
\begin{itemize}
\item
$\chi_i(g_j)\chi_j(g_i)=q_{ij}q_{ji},\ \ \chi_i(g_i)=q_{ii}$
for all $1\leq i,j \leq n.$ Here $q_{ij}=\chi(e_i,e_j)$ for $e_i,e_j\in E$.
\item $\operatorname{ht}(\alpha)<\infty$ for all $\alpha\in \Delta$.
\end{itemize}
\end{definition}

For a series of quasi-characters $(\chi_1,\cdots,\chi_n)$ of finite type associated to $\widetilde{\Phi}_{g_1},\cdots ,\widetilde{\Phi}_{g_n}$, we can attach to it a twisted Yetter-Drinfeld module $V(\chi_1,\cdots,\chi_n)$ with a canonical basis $\{X_1,\cdots, X_n\}$ such that $g_i\triangleright X_j=\chi_j(g_i)X_j$ and $\delta_L(X_i)=g_i\otimes X_i$ for all $1\leq i,j\leq n$.

\begin{theorem}
Let $G$ be a finite abelian group of odd order, $\Phi$ a $3$-cocycle on $G$.

\begin{itemize}
\item[(1)]  If $(\chi_1,\cdots,\chi_n)$ is a series of quasi-characters of finite type associated to the 2-cocycles $\widetilde{\Phi}_{g_1},\cdots ,\widetilde{\Phi}_{g_n}$, then $B(V(\chi_1,\cdots,\chi_n))$ is a finite-dimensional Nichols algebra in $_{\k G}^{\k G} \mathcal{YD}^\Phi$.
\item[(2)] Suppose that $\mathcal{C}$ is a coradically graded pointed finite tensor category such that $G(\mathcal{C})=G$ and the associator is $\Phi$. Then there exists a series of quasi-characters $(\chi_1,\cdots,\chi_n)$ of finite type associated to $\widetilde{\Phi}_{g_1},\cdots ,\widetilde{\Phi}_{g_n}$ such that $$\mathcal{C}\cong \comod( B(V(\chi_1,\cdots,\chi_n))\#\k G).$$
\end{itemize}
\end{theorem}
\begin{proof}
\begin{itemize}
\item[(1)] Direct consequence of Proposition \ref{p3.13}.
\item[(2)] As $\mathcal{C}$ is a coradically graded pointed finite tensor category, there is a finite-dimensional coradically graded pointed coquasi-Hopf algebra $M$ over $G$ such that $\mathcal{C}\cong \comod(M)$. Let $R$ be the coinvariant subalgebra of $M$. By Theorem \ref{t5.10},  $R=B(R[1])$ is a finite-dimensional Nichols algebra of diagonal type in $_{\k G}^{\k G}\mathcal{YD}^\Phi$. So there exists a series of quasi-characters $\chi_1,\cdots,\chi_n$ of finite type associated to $\widetilde{\Phi}_{g_1},\cdots ,\widetilde{\Phi}_{g_n}$ such that $R\cong B(V(\chi_1,\cdots,\chi_n)).$ Hence we have $\mathcal{C} \cong \comod(M)\cong \comod(B(V(\chi_1,\cdots,\chi_n))\#\k G).$
\end{itemize}
\end{proof}

\section{Further examples and problems}
%\subsection{}
So far, we have clarified the Nichols algebras of simple twisted Yetter-Drinfeld modules over an arbitrary finite abelian group. The main idea is: \emph{by considering the support group of a simple Yetter-Drinfeld module, one may transform a nondiagonal twisted Yetter-Drinfeld module to a diagonal one.}  Then the previously developed related theories, see e.g. \cite{H0, H4, HLYY, HLYY2}, can be fully applied.  Remarkably, this is already enough to help us classify the finite-dimensional coradically graded pointed coquasi-Hopf algebras over abelian groups of odd order.

%\subsection{}
To move on, we should study pointed finite tensor categories $\mathcal{C}$ over abelian groups with $2 \mid G(\mathcal{C}).$ The critical step is a thorough investigation of the Nichols algebras of nonsimple nondiagonal twisted Yetter-Drinfeld modules. Thanks to Proposition \ref{p4.19}, we may applied the same strategy to Nichols algebras of form $B(V_1 \oplus V_2),$ where $V_1$ and $V_2$ are nondiagonal simples. Thus, we will have a complete classification based on our previous work \cite{HLYY2}. As this will not provide further valuable insights than what we have done in Sections 3 and 4, we do not include a detailed discussion about this. Instead, we provide several simple examples in the following to elucidate the process.

Let $G=\Z_2\times \Z_2\times \Z_2=\langle g_1\rangle\times \langle g_2\rangle \times \langle g_3\rangle$, $\Phi$ a nonabelian $3$-cocycle on $G$ given by
$$\Phi(g_1^{i_1}g_2^{i_2}g_3^{j_3},g_1^{j_1}g_2^{j_2}g_3^{i_3},g_1^{k_1}g_2^{k_2}g_3^{k_3})=(-1)^{i_3j_2k_1}.$$
Consider three simple objects $V_1,V_2,V_3$  in $_{\k G}^{\k G}\mathcal{YD}^\Phi$ given as follows:
\begin{itemize}
\item $V_1=\k \{X_1,X_2\}$, $g_{V_1}=g_1$, $g_1\triangleright X_1=-X_1,$  $ g_1\triangleright X_2=-X_2,$ $g_2\triangleright X_1=X_1,$ $ g_2\triangleright X_2=-X_2,$ $ g_3\triangleright X_1=X_2, $ $g_3\triangleright X_2=X_1$.
\item $V_2=\k \{Y_1,Y_2\}$,  $g_{V_2}=g_2$,  $g_1\triangleright Y_1=Y_1,$  $g_1\triangleright Y_2=-Y_2,$ $g_2\triangleright Y_1=-Y_1,$ $g_2\triangleright Y_2=-Y_2,$  $g_3\triangleright Y_1=Y_2,$  $g_3\triangleright Y_2=Y_1$.
\item $V_3=\k \{Z_1,Z_2\}$,  $g_{V_3}=g_3$,  $g_1\triangleright Z_1=Z_2,$ $g_1\triangleright Z_2=Z_1,$ $g_2\triangleright Z_1=Z_1,$ $g_2\triangleright Z_2=-Z_2,$ $g_3\triangleright Z_1=-Z_1,$ $g_3\triangleright Z_2=-Z_2$.
\end{itemize}

Note that $V_1$ is isomorphic to the twisted Yetter-Drinfeld module $V$ given in Example \ref{e3.19}. Similarly, one can verify that $V_2, V_3$ are indeed simple twisted Yetter-Drinfeld modules in $_{\k G}^{\k G}\mathcal{YD}^\Phi$ with the help of Proposition \ref{p2.4}. Note that $V_1,V_2,V_3$ satisfy the condition $\mathrm{C1}$ of Proposition \ref{p4.23}. Hence $B(V_1),B(V_2)$ and $B(V_3)$ are finite-dimensional Nichols algebras. In additon, we have the following observation.

\begin{proposition}
The Nichols algebras $B(V_1\oplus V_2),B(V_1\oplus V_3)$ and $B(V_2\oplus V_3)$ are finite-dimensional.
\end{proposition}

\begin{proof}
Firstly we will show that $B(V_1\oplus V_2)$ is finite-dimensional. Since the support group of $V_1\oplus V_2$ is $G_1:=\langle g_1\rangle \times \langle g_2\rangle$, we have $\Phi|_{G_1}$ is an abelian $3$-cocycle on $G_1$ by proposition \ref{p2.26}. In the following we use the notations of Subsection \ref{sub3.2}. Let $\mathbbm{G}_1=\langle \mathbbm{g}_1\rangle\times \langle \mathbbm{g}_2\rangle$ and $\pi_1^*:\mathbbm{G}_1\to G_1$ be the group epimorphism given by $\pi_1^*(\mathbbm{g}_1)=g_1,\pi_1^*(\mathbbm{g}_2)=g_2$, and $V=V_1\oplus V_2$. By Lemma \ref{l4.9} and Proposition \ref{p3.9}, $B(V)\in {_{\k \mathbbm{G}_1}^{\k \mathbbm{G}_1}\mathcal{YD}^{\pi_1^*(\Phi|_{G_1})}}.$ As a consequence,  there is a $2$-cochain $J$ on $\mathbbm{G}_1$ such that $B(V)^J=B(V^J)\in {_{\k \mathbbm{G}_1}^{\k \mathbbm{G}_1}\mathcal{YD}}$. According to \eqref{3.4}, we have:
\begin{eqnarray*}
&& \mathbbm{g}_1\triangleright_JX_1=-X_1, \quad \mathbbm{g}_1\triangleright_JX_2=-X_2,\\
&& \mathbbm{g}_1\triangleright_JY_1=\frac{J(\mathbbm{g}_1,\mathbbm{g}_2)}{J(\mathbbm{g}_2,\mathbbm{g}_1)}Y_1, \quad \mathbbm{g}_1\triangleright_JY_2=-\frac{J(\mathbbm{g}_1,\mathbbm{g}_2)}{J(\mathbbm{g}_2,\mathbbm{g}_1)}Y_2,\\
&&\mathbbm{g}_2\triangleright_JX_1=\frac{J(\mathbbm{g}_2,\mathbbm{g}_1)}{J(\mathbbm{g}_1,\mathbbm{g}_2)}X_1, \quad
\mathbbm{g}_2\triangleright_JX_2=-\frac{J(\mathbbm{g}_2,\mathbbm{g}_1)}{J(\mathbbm{g}_1,\mathbbm{g}_2)}X_2,\\
&& \mathbbm{g}_2\triangleright_jY_1=-Y_1, \quad \mathbbm{g}_2\triangleright_jY_2=-Y_2.
\end{eqnarray*}
So the generalized Dynkin diagram corresponding to $B(V^J)$ is
\[ {\setlength{\unitlength}{1mm}
\Dchaintwo{}{$-1$}{$-1$}{$-1$}\ \ \ \ \ \
\Dchaintwo{}{$-1$}{$-1$}{$-1$}} \quad .\] This implies that $B(V^J)$ is finite-dimensional according to \cite[Table 1 of Section 3]{H4}. So $B(V_1\oplus V_2)$ is finite-dimensional.

 Next we will prove that $B(V_1\oplus V_3)$ is finite-dimensional. Let $R_1=X_1+X_2, \ R_2=X_1-X_2, \ S_1=Y_1+Y_2, \ S_2=Y_1-Y_2$. Then we have:
 \begin{eqnarray*}
 &&g_1\triangleright R_1=-R_1, \ g_1\triangleright R_2=-R_2, \ g_3\triangleright R_1=R_1, \ g_3\triangleright R_2=-R_2,\\
 &&g_1\triangleright S_1=S_1, \ g_1\triangleright S_2=-S_2, \ g_3\triangleright S_1=-S_1, \ g_3\triangleright S_2=-S_2.
 \end{eqnarray*}
 Let $G_2=\langle g_1\rangle \times \langle g_3\rangle$. Then $B(V_1\oplus V_3)$ is a Nichols algebra of diagonal type in $_{\k G_2}^{\k G_2}\mathcal{YD}^{\Phi|_{G_2}}$. The remaining steps of the proof are Similar to that of $B(V_1\oplus V_2)$. One can show that $B(V_1\oplus V_3)$ is twist equivalent to an ordinary Nichols algebra of diagonal type corresponding to the same generalized Dynkin diagram of $B(V_1\oplus V_2)$. Hence $B(V_1\oplus V_3)$ is finite-dimensional.

Similarly, one can show that the Nichols algebra $B(V_2\oplus V_3)$ is finite-dimensional as the previous two cases. We complete the proof of the proposition.
\end{proof}

%\subsection{}
For more general situation, the main challenge is the study of Nichols algebra $B(V)$ where $V$ has at least $3$ nondiagonal simple Yetter-Drinfeld summands over its support group $G_V.$ In our opinion, this is the truly new phenomenon for pointed finite tensor categories over finite abelian groups with nonabelian $3$-cocycles as associators. The main difficulty lies in a lack of tools for the investigation of nondiagonal Nichols algebras in twisted Yetter-Drinfeld categories. Moreover, it is also much more difficult to prove EGNO's conjecture in this situation. However, for the Hopf case there is already a nice theory for nondiagonal Nichols algebras developed by Andruskiewitsch, Heckenberger and Schneider in \cite{AHS}. We wish to extend this theory to quasi-Hopf case and pursue a complete classification of pointed finite tensor categories over abelian groups in the future works.

\end{document}